\documentclass[11pt]{amsart}
\usepackage{amsfonts}
\usepackage{amsfonts,latexsym,rawfonts,amsmath,amssymb,amsthm}
\usepackage[plainpages=false]{hyperref}
\usepackage[cal=boondox,scr=boondoxo]{mathalfa}

\usepackage{graphicx}

\RequirePackage{color}

\theoremstyle{plain}
  \newtheorem{theorem}{Theorem}[section]
  \newtheorem{lemma}{Lemma}[section]
  \newtheorem{proposition}{Proposition}[section]
  
  \newtheorem{corollary}{Corollary}[section]
  \newtheorem{definition}{Definition}[section]

\theoremstyle{remark} 
  \newtheorem{remark}{Remark}[section]

\renewcommand{\det}{\mbox{det}}
  
  \newcommand{\dist}{\mbox{dist}}
  \newcommand{\diam}{{\mbox{diam}}}

  \numberwithin{equation}{section}
  \numberwithin{figure}{section}

\begin{document}

\title[Nonlinear optimisation with potentials]
{A class of nonlinear optimisation and applications}

\author[Jiakun Liu]
{Jiakun Liu}

\address
	{School of Mathematics and Applied Statistics,
	University of Wollongong,
	Wollongong, NSW 2522, AUSTRALIA}
	
\email{jiakunl@uow.edu.au}

\thanks{This work is supported by the Australian Research Council, DP170100929}

\thanks{\copyright 2019 by the author. All rights reserved}



\date\today

\begin{abstract}
In this paper, we introduce a class of nonlinear optimisation problems. 
Under mild assumptions, we obtain the existence of potential functions and show that the potential function is a generalised solution of a Monge-Amp\`ere type equation.
We also present some interesting applications in optimal transportation and geometric optics problems.
\end{abstract}

\maketitle

\baselineskip=16.4pt
\parskip=3pt

\section{Introduction}\label{s1}

In this paper, we introduce a class of nonlinear optimisation problems, which extends Kantorovich's linear optimisation in the optimal transport problem. 
Let $U,V \subset \mathbb{R}^n$ be two bounded domains, and $\phi=\phi(x,y,t,s) : U \times V \times \mathbb{R}\times\mathbb{R} \to \mathbb{R}$ be a given constraint function, which is assumed to be $C^1$ smooth and strictly increasing in $t$ and $s$. 
\begin{definition}\label{dd11}
Let $(u,v)\in C(U)\times C(V)$ be a pair of continuous functions. 
We say $(u,v)$ is a \emph{dual pair} with respect to $\phi$ if it satisfies
	\begin{equation}\label{s110}
	\left\{
	\begin{array}{l}
		u(x) = \sup\{t\,:\,\phi(x,y,t,v(y))\leq0,\ \ \forall\, y\in V\}, \\
		v(y) = \sup\{s\,:\,\phi(x,y,u(x),s)\leq0,\ \ \forall\, x\in U\}.  
	\end{array}\right. 
	\end{equation}
Denote the set of above dual pairs by
	\begin{equation}\label{0003}
		K=\left\{(u,v)\in C(U)\times C(V)\,:\, (u,v) \text{ satisfies } \eqref{s110} \right\}.
	\end{equation}
\end{definition}
Since the function $\phi\in C^1$ and is strictly increasing in $t,s$, by the implicit function theorem there is a $C^1$ function $\varphi=\varphi(x,y,s)$ strictly increasing in $s$ such that the constraint $\phi\leq0$ can be written as
	\begin{equation}\label{ne01}
		\phi(x,y,u,v)=u+\varphi(x,y,v)\leq0.
	\end{equation}
We assume further that there exists a constant $\theta_0>0$ such that
		\begin{equation}\label{2004}
			\varphi_s\geq\theta_0\quad \mbox{in }\ U\times V\times\mathbb{R},
		\end{equation}
and $\varphi$ satisfies the condition 
	\begin{itemize}
	\item[($H_1$)] For each $x_0\in U$, for any $(p,t)\in\mathbb{R}^n\times\mathbb{R}$, there is at most one pair $(y,s)\in\mathbb{R}^n\times\mathbb{R}$ such that
		\[\left(\varphi_x,\varphi\right)(x_0,y,s)=-(p,t),\]
namely, $\varphi_x(x_0,y,s)=-p$ and $\varphi(x_0,y,s)=-t$.
And, for each $y_0\in V$, for any $(q,s)\in\mathbb{R}^n\times\mathbb{R}$, there is at most one pair $(x,t)\in\mathbb{R}^n\times\mathbb{R}$ such that
		\[\left(\frac{\varphi_y}{\varphi_s},\varphi\right)(x,y_0,s)=-(q,t),\]
namely, $(\varphi_y/\varphi_s)(x,y_0,s)=-q$ and $\varphi(x,y_0,s)=-t$.
	\end{itemize}

\vskip5pt

Under above assumptions, we have the following result, whose proof is postponed in \S\ref{s2}. 
\begin{lemma}\label{p202} 
For each dual pair $(u,v)\in K$, there exists an associated mapping $T : U\to V$ that solves the equation  
 	\begin{equation}\label{as90new}
  		u(x) + \varphi(x,Tx, v(Tx)) = 0
 	\end{equation}
and is uniquely determined almost everywhere on $U$.
Meanwhile, there exists a mapping $T^{-1} : V\to U$ solving
 	\begin{equation}\label{as90new1}
  		u(T^{-1}y) + \varphi(T^{-1}y, y, v(y)) = 0
 	\end{equation}
and is uniquely determined almost everywhere on $V$.
Moreover, $T^{-1}(Tx) = x$, a.e. on $U$ and $T(T^{-1}y) = y$, a.e. on $V$. 
\end{lemma}

For $(u,v)\in K$, define the functional
	\begin{equation}\label{0001}
		I(u,v):=\int_U u(x)f(x)\,dx + \int_V \varphi(T^{-1}y, y, v(y))g(y)\,dy,
	\end{equation}
where $f, g$ are probability density functions on $U, V$, respectively, satisfying  
	\begin{equation}\label{ne06}
		\int_U f(x)dx = \int_V g(y)dy = 1.
	\end{equation}	
In this paper, we study the nonlinear optimisation problem of 
\begin{equation}\label{prob}
	\text{maximising $I(u,v)$ in the constraint set $K$. }
\end{equation}
More generally, one could consider $U, V$ as subsets of a Riemannian manifold $\mathcal{M}$ with general measures $\mu, \nu$, but for simplicity, here we only consider the Euclidean case with absolutely continuous measures $\mu=f\,dx,\ \nu=g\,dy$. 
Some interesting examples and applications of this problem are contained in \S\ref{s5}.

Our main result is the following solvability of the nonlinear optimisation problem \eqref{prob}.

\begin{theorem}\label{t003}
Under the hypotheses \eqref{2004}, \eqref{ne06} and ($H_1$), there exists a dual maximising pair $(u,v)\in K$ of $I$, and the mapping $T:U\to V$ associated to $(u,v)$ is uniquely determined almost everywhere on $U$.
\end{theorem}
The functions $u,v$ in a dual maximising pair are called \emph{potential functions} of the nonlinear optimisation \eqref{prob}.
The associated mapping $T$ is called \emph{optimal mapping}.
These terminologies are adopted from optimal transportation \cite{MTW,U,V}.
However, it is worth pointing out that in the nonlinear case, there is generally no uniqueness for maximising pair $(u,v)$, see Remark \ref{rnl2}.

\vskip5pt

It is well known that many constrained nonlinear optimisation problems can be solved by Lagrangian dual methods (see, for example, \cite{GY}), where the convexity plays a crucial role. 
For the nonlinear optimisation \eqref{prob} under conditions \eqref{2030}--\eqref{2031}, we can show that there is no duality gap and there exists at least one Lagrange multiplier. This enables us to use the Lagrangian duality theory to study the maximisation of the functional $I$, in \S\ref{s4}.  

In order to state the result, let us introduce some terminology in Lagrangian duality. More details are contained in \S\ref{s4}.
Denote $X:=C(U)\times C(V)$. The nonlinear optimisation \eqref{prob} is a special case of the following \emph{primal problem}:
	\begin{eqnarray}
		&&\mbox{maximise } I(u,v)=\int_{U\times V}F(x,y,u(x),v(y))\,d\gamma, \label{0001new} \\ 
		&&\mbox{subject to } (u,v)\in X, \quad \psi(u,v):=\inf_{x\in U, y\in V}-\phi(x,y,u(x),v(y))\geq0, \nonumber	
	\end{eqnarray}
where $F$ is a function in $\mathbb{R}^{2n+2}$, and $d\gamma$ is a measure on $U\times V$ with $dx, dy$ as its marginals. 
We always assume that $\phi$ has the form \eqref{ne01}. 
An element $(u,v)\in X$ is called \emph{feasible} if $\psi(u,v)\geq0$.

Define the \emph{Lagrangian function} $L : X\times\mathbb{R}\to\mathbb{R}$ to be
	\begin{equation}\label{2024}
		L(u,v,\mu) = I(u,v)+\mu \psi(u,v),
	\end{equation}
where $\mu\in\mathbb{R}$. The \emph{dual functional} $J$ is defined by
	\begin{equation}\label{2028}
		J(\mu)=\sup_{(u,v)\in X}L(u,v,\mu),
	\end{equation}
and the \emph{dual problem} is given by	
	\begin{equation}\label{2029}
	\begin{split}
		&\mbox{minimise } J(\mu)\\
		&\mbox{subject to } \mu\geq0.		
	\end{split}
	\end{equation}
Regardless of the functional $I$ and the constraint $\phi$ of the primal problem, the dual problem has a very nice convexity property, as shown in Lemma \ref{la04}. In the language of nonlinear programming \cite{Ber,GY}, when $\inf_{\mu\geq0}J(\mu)=\sup_{(u,v)\in K}I(u,v)$, we say that \emph{there is no duality gap}, otherwise, \emph{there is duality gap}.

\begin{theorem}\label{t004}
Assume that the function $F$ in \eqref{0001new} is concave in $(t,s)$, namely for any $(x,y,t,s)\in U\times V\times\mathbb{R}^2$, 
		\begin{equation}\label{2030}
			Hess_{t,s}F:=\left(\begin{array}{ll}
			\partial_{tt}F,&\partial_{ts}F \\
			\partial_{st}F,&\partial_{ss}F
			\end{array}\right)\leq0,
		\end{equation}
and the constraint $\varphi$ in \eqref{ne01} is convex in $s$, namely for any $(x,y,s)\in U\times V\times \mathbb{R}$, 
		\begin{equation}\label{2031}
			\partial_{ss}\varphi\geq0.
		\end{equation}	
Suppose that there exists a pair $(\bar u,\bar v)\in X$ such that 
		\begin{equation}\label{2032}
			\psi(\bar u,\bar v)>0.
		\end{equation}
Then there is no duality gap between the primal problem \eqref{0001new} and dual problem \eqref{2029}, and there exists at least one Lagrange multiplier, (see Definition \ref{d301}).
\end{theorem}

Note that in the special case \eqref{0001}, the convexity assumptions \eqref{2030} and \eqref{2031} imply that $\varphi=\varphi(x,y,v)$ is linear in $v$, namely 
	\[\varphi(x,y,v)=c_0(x,y)+c_1(x,y)v,\]
for some functions $c_0,c_1$, where $c_1\geq\theta_0>0$ in $U\times V$ by the monotonicity \eqref{2004}. 
When $c_1\equiv1$, it is optimal transportation with the associated cost function $-c_0$, see Example \ref{ss51}.

This paper is organised as follows: In \S\ref{s2}, we prove Lemma \ref{p202}, Theorem \ref{t003}, and show that the optimal mapping $T$ is measure preserving in the sense of \eqref{jy11} and the potential function $u$ satisfies a Monge-Amp\`ere type equation.  
In \S\ref{s3}, we introduce a notion of generalised solutions and show that a potential function is a generalised solution, from which the existence of generalised solutions follows, see Theorem \ref{the1}.
In \S\ref{s4}, it contains the Lagrangian duality theory, which provides a useful tool to obtain the existence of maximiser. Theorem \ref{t004} is proved. 
In \S\ref{s5}, we present some examples and applications of the nonlinear optimisation \eqref{prob}. 
In particular, we derive the equations in geometric optics problems from the corresponding constraints, instead of using the reflection or refraction law.

\vskip5pt

\noindent\textit{Acknowledgements.}
The author would like to thank Professor Xu-Jia Wang for proposing the problem to him.
He would also like to thank Professor Neil Trudinger for many discussions on this topic. 
This is a revision of the previous version of arXiv1203.2351.

\section{Existence of maximisers}\label{s2}

In this section, we obtain the solvability of the nonlinear optimisation problem \eqref{prob}. 
Firstly, we shall prove Lemma \ref{p202} so that the mappings $T$ and $T^{-1}$ are well-defined. 

\begin{proof}[Proof of Lemma \ref{p202}]
Since $u$ satisfies \eqref{s110} and $v,\varphi$ are continuous, for each $x\in U$, there exists some $y=:T(x)\in\overline{V}$ such that
	\begin{align}
		u(x)+\varphi(x,y,v(y)) &= 0, \label{ne03}\\
		u(x')+\varphi(x',y,v(y)) &\leq 0, \nonumber
	\end{align}
for any other $x'\in U$. Note that since $\varphi$ is $C^1$ smooth, one can see $u$ is locally Lipschitz, and thus differentiable almost everywhere. 
Let $x\in U$ be a differentiable point of $u$, by differentiation we have
	\begin{equation}\label{2014}
	  \varphi_x(x,y,v)+Du(x)=0.
	\end{equation}

Therefore, for the fixed $x\in U$, setting $t=u(x)$ and $p=Du(x)$ one can see that
	\[\varphi_x(x,y,v)=-p,\quad\mbox{and }\varphi(x,y,v)=-t.\]
From the condition ($H_1$), we then obtain the mapping $y=T(x)$ solving the equation \eqref{as90new}. 
Since $u$ is differentiable almost everywhere on $U$, the mapping $T$ is uniquely determined almost everywhere on $U$.

Similarly, we can obtain the mapping $T^{-1}$ that solves the equation \eqref{as90new1} and is uniquely determined almost everywhere on $V$. 
It remains to show that $T$ and $T^{-1}$ is essentially inverse to each other. In fact, this follows from the unique solvability of equations \eqref{as90new} and \eqref{as90new1}. For example, setting $x=T^{-1}y$ in \eqref{as90new1}, one has
	$$ u(x)+\varphi(x,y,v(y))=0. $$
On the other hand, the above equation is uniquely solved by $y=Tx$ a.e.. Hence, $T(T^{-1}y)=y$ for a.e. $y\in V$. Similarly, one can see that $T^{-1}(Tx)=x$ for a.e. $x\in U$.  
\end{proof}

To prove the existence result in Theorem \ref{t003}, we would simplify the notation a bit.
As mentioned before, the functional $I$ in \eqref{0001} is a special case of \eqref{0001new} if the function $F$ is chosen by
	\begin{equation}\label{newF}
		F(x,y,u(x),v(y)) = \frac{1}{|V|}u(x)f(x)+\frac{1}{|U|}\varphi(T^{-1}y,y,v(y))g(y).
	\end{equation}
In the following content, we always assume $F$ is given by \eqref{newF}. 

\begin{proof}[Proof of Theorem \ref{t003}]
Given any pair $(u,v)\in C(U)\times C(V)$ satisfying $\phi(x,y,u,v)\leq 0$ in $U\times V$, we claim that $I(u,v)$ does not decrease if $v$ is replaced by
	\begin{equation}\label{2007}
		v^*(y) = \sup\{s\,:\,\phi(x,y,u(x),s)\leq0,\ \ \forall\, x\in U\}.
	\end{equation}
This means it is sufficient to consider the maximisation of $I$ in the constraint set $K$. 

In fact, by the continuity of $\phi$ and $u$, for each $y\in V$ there is some $x\in\overline{U}$ such that
	\[\phi(x,y,u(x),v^*(y))=0\geq\phi(x,y,u(x),v(y)),\]
where the last inequality follows from the assumption $\phi(x,y,u,v)\leq 0$ in $U\times V$. By \eqref{ne01}--\eqref{2004}, $v^*\geq v$. Furthermore, $\phi(x,y,u(x),v^*(y))\leq0$ for all $(x,y)\in U\times V$.

Since $v^*\geq v$, by \eqref{newF} we have
	\[I(u,v^*)\geq I(u,v).\]
Similarly, if we define
	\begin{equation}\label{2006}
		u^*(x) = \sup\{t\,:\,\phi(x,y,t,v^*(y))\leq0,\ \ \forall\, y\in V\}, 
	\end{equation}
then  $\phi(x,y,u^*(x),v^*(y))\leq0$ in $U\times V$, and
	\[I(u^*,v^*)\geq I(u,v^*)\geq I(u,v).\]
Thus we do not decrease $I(u,v)$ by replacing $(u,v)$ by $(u^*,v^*)$. The claim is proved. 

Define $K_{C_0}=K\cap\{u\geq C_0\}$, where $C_0$ is a constant, which may be chosen negative and sufficiently small in the following context. We show that $u^*$ and $v^*$ are uniformly bounded if $(u,v)\in K_{C_0}$. Since $v^*\geq v, u\geq C_0$, by \eqref{ne01}--\eqref{2004} we have for each $y\in V$, $s:=v^*(y)$,
	\[C_0+\varphi(x,y,s)\leq u(x)+\varphi(x,y,s)\leq 0,\quad\mbox{for all }x\in U.\]
Then by \eqref{2004} again, there exists a constant $C_1$, such that $s\leq C_1$. This implies that
	\begin{equation}
		v\leq v^*\leq C_1,
	\end{equation}
we may choose $C_1$ such that $\sup_Vv^*=C_1$. By a similar argument, there is another constant $\tilde C_0$ depending on $\varphi$ and $C_1$ such that $\inf_Uu^*=\tilde C_0$. The constant $\tilde C_0\geq C_0$, since $u^*\geq u$ in $U$, and so $(u^*,v^*)\in K_{C_0}$.

We next deduce the lower bound of $v^*$ and the upper bound of $u^*$ by showing that $u^*$ and $v^*$ are locally Lipschitz. Consider two points in $U$, $x_1\neq x_2$ and $|x_1-x_2|<\varepsilon$ sufficiently small. There are two points $y_1,y_2\in\overline{V}$ such that
	\begin{eqnarray*}
		\phi(x_1,y_1,u^*(x_1),v^*(y_1))\!\!&=&\!\!0, \\
		\phi(x_2,y_2,u^*(x_2),v^*(y_2))\!\!&=&\!\!0.
	\end{eqnarray*}
Then by \eqref{ne01}, we have
	\[\begin{split}
		0 =& \phi(x_2,y_2,u^*(x_2),v^*(y_2))-\phi(x_1,y_2,u^*(x_1),v^*(y_2)) \\
		   &+\phi(x_1,y_2,u^*(x_1),v^*(y_2))-\phi(x_1,y_1,u^*(x_1),v^*(y_1)) \\
		  =& u^*(x_2)-u^*(x_1)-\varphi_x(\hat x,y_2,v^*(y_2))\cdot(x_2-x_1) \\
		   &+\phi(x_1,y_2,u^*(x_1),v^*(y_2)),
	\end{split}\]
where $\hat x=\theta x_1+(1-\theta)x_2$ for some $\theta\in(0,1)$. Noting that $\phi(x_1,y_2,u^*(x_1),v^*(y_2))\leq0$, we have
	\[u^*(x_2)-u^*(x_1) \geq -C_2|x_2-x_1|,\]
where the constant $C_2=\sup(|\varphi_x|+|\varphi_y|)$. 

On the other hand, replacing $\phi(x_1,y_2,u^*(x_1),v^*(y_2))$ by $\phi(x_2,y_1,u^*(x_2),v^*(y_1))$ in the above calculation, we have
	\[u^*(x_2)-u^*(x_1) \leq C_2|x_2-x_1|.\]
Therefore, the Lipschitz constant of $u^*$ on $U$ is controlled by 
	\begin{equation}\label{2010}
		\|u^*\|_{Lip(U)}\leq C_2.
	\end{equation}
By switching $x$ and $y$	in the above argument, we can obtain the Lipschitz continuity of $v^*$ on $V$,
	\[\begin{split}
		|v^*(y_2)-v^*(y_1)| &\leq \frac{\sup|\varphi_y|}{\inf \varphi_s}|y_2-y_1| \\
		 				   &\leq C_2\theta_0^{-1}|y_2-y_1|, 
	\end{split}\]
where $\theta_0$ is the constant in \eqref{2004}, $y_1, y_2$ are two distinct points in $V$. This inequality implies that $\|v^*\|_{Lip(V)}\leq C_2\theta_0^{-1}$. Hence, we have $u^*\leq \tilde C_0+C_2\diam(U)$ and $v^*\geq C_1-C_2\theta_0^{-1}\diam(V)$.

We conclude, therefore, that any pair $(u,v)\in K_{C_0}$ may be replaced by a bounded, Lipschitz pair $(u^*,v^*)\in K_{C_0}$ without decreasing $I$. We now choose a maximising sequence $\{(u_k,v_k)\}\subset K_{C_0}$ such that 
	\[I(u_k,v_k)\to\sup_{(u,v)\in K_{C_0}}I(u,v).\] 
By the above considerations we may assume that each $(u_k,v_k)$ is a bounded, uniformly Lipschitz pair, uniformly with respect to $k$,
so there is a subsequence converging uniformly to a bounded, Lipschitz, maximising pair $(\bar u,\bar v)\in K_{C_0}$.

Last, we show that when $C_0<0$ sufficiently small,
	\[\sup_{(u,v)\in K_{C_0}}I(u,v)=\sup_{(u,v)\in K}I(u,v),\]
or equivalently, $\sup_{K_{C_0}}I$ is independent of $C_0$. By definition, one has $\sup_{K_{C_0-1}}I\geq\sup_{K_{C_0}}I$. So it suffices to show the reverse inequality. 
Let $(u,v)\in K_{C_0-1}$ be a maximiser such that $I(u,v)=\sup_{K_{C_0-1}}I$, and $\{x_k\}_{k=1,\cdots,N}$ be a set of points in $U$. For a small constant $\varepsilon>0$, define
	\[\tilde u=\left\{
	\begin{array}{ll}
		u & \mbox{in }U-\cup_NB_\varepsilon(x_k), \\
		u+2 & \mbox{in }\cup_NB_\varepsilon(x_k).
	\end{array}\right.\]
Note that we may replace $\tilde u$ by its mollification $\tilde u_h=\rho_h*\tilde u$, where $\rho_h$ is the standard mollifier function \cite{GT}. For simplicity, we assume $\tilde u$ continuous in the sense that for $h>0$ sufficiently small,
	\[I(\tilde u_h,v)=I(u,v)+O(N\varepsilon^n).\]
Define
	\begin{eqnarray*}
		\tilde v^*(y) \!\!&=&\!\! \sup\{s\,:\,\phi(x,y,\tilde u(x),s)\leq0,\ \ \forall\, x\in U\}, \\
		\tilde u^*(x) \!\!&=&\!\! \sup\{t\,:\,\phi(x,y,t,\tilde v^*(y))\leq0,\ \ \forall\, y\in V\}.
	\end{eqnarray*}
Since the constraint function $\varphi$ is $C^1$ smooth in $s$ and by \eqref{ne01}--\eqref{2004}, except a set $E\subset U$ and a set $E'\subset V$ of measure $|E|=|E'|=O(N\varepsilon^n)$,
	\begin{eqnarray*}
		\tilde v^* \!\!&=&\!\! v-\frac{2}{\varphi_s}+O(\delta)\quad\mbox{in }V\setminus E', \\
		\tilde u^* \!\!&=&\!\! u+2+O(\delta)\quad\mbox{ in }U\setminus E,
	\end{eqnarray*}
where $\delta:=\min_{i\neq j}\{\dist(x_i,x_j)\}$. Therefore, by \eqref{ne06}
and the mean value theorem we have
	\[\begin{split}
		I(\tilde u^*,\tilde v^*) &= I(u,v)+2\int_{(U\setminus E)\times(V\setminus E')}\left\{F_t-\frac{F_s}{\varphi_s}\right\}d\gamma+O(\delta)+O(N\varepsilon^n) \\
		&\geq I(u,v)-C\delta-CN\varepsilon^n.
	\end{split}\]
As $(u,v)\in K_{C_0-1}$, we may assume that $\inf_U u=C_0-1$. 
Otherwise, one has $\inf_Uu=C_0-\tau_0$ for some constant $\tau_0<1$. 
This implies that $\sup_{K_{C_0-1}}I=\sup_{K_{C_0}}I$, namely $\sup_{K_{C_0}}I$ is independent of $C_0$, and the proof is finished. 
By the definition, $\delta$ will become small if the number of points $N$ is sufficiently large so that we have $(\tilde u^*,\tilde v^*)\in K_{C_0}$ and
	\[\sup_{K_{C_0}}I\geq I(\tilde u^*,\tilde v^*)\geq\sup_{K_{C_0-1}}I-C\delta-CN\varepsilon^n.\]
Then, choosing $\varepsilon>0$ sufficiently small, we have
	\[\sup_{K_{C_0-1}}I\leq\sup_{K_{C_0}}I,\]	
by letting $\delta\to0, \varepsilon\to0$, which implies that $\sup_{K_{C_0}}I$ is independent of $C_0$, and the proof is finished.
\end{proof}

\begin{remark}\label{rnl2}
From the above proof of Theorem \ref{t003}, we conclude that there exist infinitely many maximising pairs. In fact, if $(u,v)$ is a maximiser and $C_0=\inf_U u$, then there is another maximiser in ${K}_{C_0+1}$, which is different from $(u,v)$.
\end{remark}

\begin{remark}\label{rrr1}
In the condition ($H_1$), we assume that $(\varphi_x,\varphi)(x,\cdot,\cdot)$ is one-to-one in the whole space $\mathbb{R}^n\times\mathbb{R}$, for each $x\in U$. This is only for simplicity. We may allow that the constraint $\varphi$ is defined in a proper subset $\mathcal{U}\times\mathcal{I}$, where $\mathcal{U}\subset\mathbb{R}^n\times\mathbb{R}^n$ and $\mathcal{I}\subset\mathbb{R}$. Denote the projections $\mathcal{U}_x=\{y\in\mathbb{R}^n\,:\,(x,y)\in\mathcal{U}\}$, $\mathcal{U}_y=\{x\in\mathbb{R}^n\,:\,(x,y)\in\mathcal{U}\}$, let $U=\bigcup\mathcal{U}_y$, $V=\bigcup\mathcal{U}_x$. 
In this case we replace ($H_1$) by assuming that: for any $(x,y)\in\mathcal{U}$, there exists an open interval $\mathcal{I}(x,y)\subset\mathcal{I}$, such that $(\varphi_x,\varphi)(x,\cdot,\cdot)$ is one-to-one in $y\in\mathcal{U}_x, s\in\mathcal{I}(x,y)$, for each $x\in U$. 
Accordingly, the injectivity of $((\varphi_y/\varphi_s),\varphi)(\cdot,y,\cdot)$ can be restricted to $\mathcal{U}\times\mathcal{I}$ in a similar way. 
\end{remark}

\begin{definition}\label{mpmap}
A mapping $S : U\to V$ is called \emph{measure preserving} if for any $h\in C(V)$,
  \begin{equation}\label{jy11}
  	\int_{U} h(Sx)f(x)\,dx = \int_{V} h(y)g(y)\,dy.
  \end{equation}
\end{definition}

\begin{lemma}\label{le11}
Assume that the balance condition \eqref{ne06} holds. Let $T$ be the optimal mapping obtained in Lemma \ref{p202}, associated with a dual maximising pair $(u,v)$.
Then $T$ is measure preserving in the sense of \eqref{jy11}. 
\end{lemma}

\begin{proof}
Let $h\in C(V)$ and $\eta(x):=h(Tx)$. For simplicity, we may assume $T, T^{-1}$ are continuous. (Actually, since $u$ is Lipschitz continuous, the assumption holds everywhere except a zero measure set, but this does not affect the integrals in \eqref{jy11}.) Hence, $\eta\in C(U)$ is a continuous function. 
Let $|\epsilon|<1$ sufficiently small. Define
	\begin{equation}\label{1009}
		 u_\epsilon(x) = u(x)+\epsilon\eta(x)
	\end{equation}
and
	\begin{equation}\label{1010}
		v_\epsilon(y) = \sup\{s\,:\,u_\epsilon(x)+\varphi(x,y,s)\leq0,\ \ \forall\, x\in U\}.
	\end{equation}
Then $I(\epsilon):=I(u_\epsilon,v_\epsilon)$ attains its maximum at $\epsilon=0$, and $(u_0,v_0)=(u,v)$. 

Since $(u,v)$ satisfies \eqref{s110}, by Lemma \ref{p202} for $y\in V$ the supremum \eqref{s110} is attained at point $x_0=T^{-1}(y)$. We claim that at these points,
	\begin{equation}\label{1011}
		\begin{split}
		v_\epsilon(y)-v(y) &= -\epsilon\frac{\eta}{\varphi_s}(x_0,y,v(y))+o(\epsilon) \\
			&= -\epsilon\frac{\eta}{\varphi_s}(T^{-1}y,y,v(y))+o(\epsilon).
	\end{split}
	\end{equation}
	
To prove \eqref{1011}, first we show that $LHS\leq RHS$.
	\[\begin{split}
		0&=u(x_0)+\varphi(x_0,y,v(y)) \\
		 &=u_\epsilon(x_0) - \epsilon\eta(x_0) + \varphi(x_0,y,v(y)) \\
		 &\leq -\varphi(x_0,y,v_\epsilon(y)) - \epsilon\eta(x_0) + \varphi(x_0,y,v(y))  \\
		 &= -\varphi_s(x_0,y,\hat v) (v_\epsilon(y)-v(y)) - \epsilon\eta(x_0),
	\end{split}\]
where $\hat v=(1-\theta)v(y)+\theta v_\epsilon(y)$ for some $\theta\in(0,1)$, and $\hat v \to v(y)$ as $\epsilon\to0$. 
Thus we have
	\[\begin{split}
		 v_\epsilon(y) - v(y) &\leq -\epsilon\frac{\eta(x_0)}{\varphi_s(x_0,y,\hat v)} \\
		 	&= -\epsilon\frac{\eta(x_0)}{\varphi_s(x_0,y,v(y))} + \epsilon\left(\frac{\eta(x_0)}{\varphi_s(x_0,y,v(y))}- \frac{\eta(x_0)}{\varphi_s(x_0,y,\hat v)}\right) \\
		 	&= -\epsilon\frac{\eta(x_0)}{\varphi_s(x_0,y,v(y))} + o(\epsilon),
	\end{split}\]
where the last equality holds since $\varphi\in C^1$ and $\hat v\to v(y)$ as $\epsilon\to0$. 	
	
To show $LHS\geq RHS$ we use the fact that for any such $y\in V$ there are points $x_\epsilon\in\overline U$ such that the supremum in \eqref{1010} is attained. Thus
	\[\begin{split}
		0&\geq u(x_\epsilon)+\varphi(x_\epsilon,y,v(y)) \\
		 &=u_\epsilon(x_\epsilon) - \epsilon\eta(x_\epsilon) + \varphi(x_\epsilon,y,v(y)) \\
		 &= u_\epsilon(x_\epsilon) + \varphi(x_\epsilon, y, v_\epsilon(y)) - \epsilon\eta(x_\epsilon) + \varphi_s(x_\epsilon, y, \bar v) (v(y)-v_\epsilon(y)),
	\end{split}\]
where $\bar v=(1-\theta)v(y)+\theta v_\epsilon(y)$ for some $\theta\in(0,1)$, and $\bar v\to v(y)$ as $\epsilon\to0$.
Then we have
	\[ v_\epsilon(y) - v(y) \geq -\epsilon\frac{\eta(x_\epsilon)}{\varphi_s(x_\epsilon, y, \bar v)}. \]
Since the supremum in \eqref{s110} is attained at $x_0$, we have $x_\epsilon\to x_0$ as $\epsilon\to0$, and since $\eta, \varphi_s$ are continuous and $\bar v\to v(y)$ as $\epsilon\to0$, therefore, we obtain
	\[\epsilon \left( \frac{\eta(x_\epsilon)}{\varphi_s(x_\epsilon, y, \bar v)} - \frac{\eta(x_0)}{\varphi_s(x_0, y, v(y))} \right)=o(\epsilon).\]
This implies that $LHS\geq RHS$, and \eqref{1011} follows.
	
Next, since $(u,v)=(u_0,v_0)$ is a maximiser of $I$ with $F$ given by \eqref{newF}, we obtain
	\[\begin{split}
		0&=\lim_{\epsilon\to0}\frac{I(u_\epsilon,v_\epsilon)-I(u,v)}{\epsilon} \\
		 &=\int_U \eta(x)f(x)\,dx - \int_V g(y)\eta(T^{-1}y)\,dy.
	\end{split}\]
Recall that $\eta(x)=h(Tx)$ and $\eta(T^{-1}y)=h(y)$, we have
	$$ \int_U h(Tx)f(x)\,dx = \int_V h(y)g(y)\,dy. $$
\end{proof}

As a consequence, we have the following

\begin{corollary}\label{cc22}
Assume the function $F$ is given by \eqref{newF} and the condition \eqref{ne06} holds. 
If the optimal mapping $T$ is continuous differentiable, then
	\begin{equation}\label{ne11}
		|\det\,DT|=\frac{f}{g\circ T}.
	\end{equation}
\end{corollary}

\begin{proof}
From Lemma \ref{le11}, one has $T$ is measure preserving in the sense of \eqref{jy11}. 
When $T$ is $C^1$ smooth, by the formula of change of coordinates,
	\[\int_Uf(x)h(Tx)dx = \int_Ug(Tx)h(Tx)|\det\,DT|dx,\]
for any $h\in C(V)$. Hence the Jacobian of $DT$ satisfies \eqref{ne11}.	 
\end{proof}

As a consequence of Lemma \ref{le11} and Corollary \ref{cc22}, we derive the equation satisfied by the potential function $u$ as follows. At this stage, let us assume all the functions are smooth enough, say at least $C^2$, so that we can do the differentiations. 

Let $(u,v)\in K$ be a dual maximising pair of $I$, and $T$ be the associated optimal mapping. 
By \eqref{2014} we have
	\[\varphi_x(x,Tx,v(Tx))+Du(x)=0,\]
in $U$. Differentiating with respect to $x$, we then get
	\begin{equation}\label{2019}
		0= \varphi_{xx}+\varphi_{xy}DT+(\varphi_{xs}\otimes Dv)DT + D^2u,
	\end{equation}
where each side is regarded as an $n\times n$ matrix valued at $(x,y)$, $y=T(x)$.

In order to eliminate $Dv$ in \eqref{2019}, we note that for $x_0\in U$ equality \eqref{as90new} holds at $y_0=T(x_0)$, and for other $y'\in V$
	\[u(x_0)+\varphi(x_0,y',v(y'))\leq0,\]
since $(u,v)\in K$. Thus, at $(x_0,y_0)$ there holds
	\[\frac{d\varphi}{dy}=\varphi_y+\varphi_sDv=0.\]
By the assumption \eqref{2004}, we thus get
	\begin{equation}\label{2020}
		Dv=-\frac{\varphi_y}{\varphi_s}.
	\end{equation}
	
Combining \eqref{2019} and \eqref{2020} we have the equation
	\begin{equation}
		\left|D^2u+\varphi_{xx}\right|=\left|\varphi_{xy}-\frac{1}{\varphi_s}\varphi_{xs}\otimes\varphi_y\right|\left|DT\right|.
	\end{equation}
In the case of \eqref{newF}, by Corollary \ref{cc22} we obtain the equation
	\begin{equation}\label{2pde}
		\left|\det\,\left[D^2u+\varphi_{xx}\right]\right|=\left|\det\,\left[\varphi_{xy}-\frac{1}{\varphi_s}\varphi_{xs}\otimes\varphi_y\right]\right|\frac{f}{g\circ T},
	\end{equation}
which is a Monge-Amp\`ere type equation \cite{Fig,GT,Gu}. 
Correspondingly, we have the natural boundary condition
	\begin{equation}\label{bpde}
		T(U)=V.
	\end{equation}
Similarly, one can also derive the dual PDE for the dual potential $v$. 

When $\varphi(x,y,v)=x\cdot y$, then \eqref{2pde} is equivalent to the standard Monge-Amp\`ere equation
	\[\det\,D^2u=h,\]
with the boundary condition
	\[Du(U)=V,\]
where $h=f/g$. When $\varphi(x,y,v)=v-c(x,y)$ for some function $c : U\times V\to\mathbb{R}$, we have the optimal transportation equation, see Example \ref{ss51},
	\[\det\,\left[D^2u-D^2_xc\right]=|\det\,D^2_{xy}c|\frac{f}{g\circ T},\]
with the boundary condition \eqref{bpde}.

It is well-known that the regularity of equation \eqref{2pde} depends crucially on the structure of constraint function $\varphi$, as in \cite{MTW,W96}. In the next section, we shall introduce a notion of generalised solution and show that if $(u,v)$ is a dual maximising pair of $I$ over $K$, then the potential $u$ is a generalised solution of \eqref{2pde}. 
The regularity property of $u$ is related to that of generated Jacobian equations, which has been systematically studied by Trudinger in \cite{T1,T2}.
Under some structural conditions (which are analogous to the MTW condition in optimal transportation \cite{MTW}), in \cite{T1,T2,GK,JT} the authors started to develop a theory parallel to that in optimal transportation.

\section{generalised solution}\label{s3}

In this section, we introduce a notion of generalised solutions of \eqref{2pde} and show that the potential function is a generalised solution.
Let $\varphi$ be the constraint in \eqref{ne01}.
First we introduce the $\varphi$-concavity for functions, which is an extension of the $c$-concavity in optimal transportation, see \cite{GM1,MTW}. 

\begin{definition}\label{dgs1}
A $\varphi$-support function of $u$ at $x_0$ is a function of the form $\varphi(x,y_0,s_0)$, where $y_0\in\mathbb{R}^n$, and $s_0\in\mathbb{R}$ is a constant such that
	\begin{eqnarray}
		u(x_0)+\varphi(x_0,y_0,s_0)\!\!&=&\!\!0, \nonumber\\
		u(x)+\varphi(x,y_0,s_0)\!\!&\leq&\!\! 0 \quad\forall\, x\in U. \label{gs01}
	\end{eqnarray}
	
A continuous function $u$ defined on $\overline{U}$ is $\varphi$-concave if for any point $x_0\in U$, there exists a $\varphi$-support function at $x_0$. 
\end{definition}

By definition, the potential function $u$ is $\varphi$-concave with $y_0\in {V}, s_0=v(y_0)$.
In the special case when $\varphi(x,y,s)=s-x\cdot y$, the notion of $\varphi$-concavity coincides with that of concavity, and the graph of a $\varphi$-support function is a support hyperplane. 

Recall that $\varphi$ is derived from the constraint function $\phi(x,y,u,v)$ by the strict monotonicity in $u$. Since $\phi$ is also strictly increasing in $v$, the constraint \eqref{ne01} can also be written as
	\begin{equation}\label{gs02}
		\phi(x,y,u,v)=v+\varphi^*(x,y,u)\leq 0,
	\end{equation}
for a function $\varphi^*=\varphi^*(x,y,t)$ strictly increasing in $t$. The function $\varphi^*=\varphi^*(x,y,t)$ is called \emph{dual constraint function} of $\varphi$ in the sense of
	\begin{align*}
		-\varphi(x,y,-\varphi^*(x,y,t)) &= t, \\
		-\varphi^*(x,y,-\varphi(x,y,s)) &= s,
	\end{align*}
for all $(x,y)\in U\times V$. By differentiating, we have
	\begin{equation}\label{gs03}
		\varphi^*_t=\frac{1}{\varphi_s},\quad \varphi^*_x=\frac{\varphi_x}{\varphi_s},\quad \varphi^*_y=\frac{\varphi_y}{\varphi_s}.
	\end{equation}
For the dual constraint $\varphi^*$, from \eqref{gs03} and the condition ($H_1$) we have:
\begin{itemize}
	\item[($H_1^*$)] For each $y_0\in V$, for any $(q,s)\in\mathbb{R}^n\times\mathbb{R}$, there is at most one pair $(x,t)\in\mathbb{R}^n\times\mathbb{R}$ such that
	\[\left(\varphi^*_y,\varphi^*\right)(x,y_0,t)=-(q,s),\]
namely, $\varphi^*_y(x,y_0,t)=-q$ and $\varphi^*(x,y_0,t)=-s$.
\end{itemize}
	
The $\varphi$-concavity in Definition \ref{dgs1} and \eqref{gs01}--\eqref{gs03} are generalisations of $c$-concavity and $c$-duality in optimal transportation, where 
	\[\varphi(x,y,s)=s-c(x,y),\quad \varphi^*(x,y,t)=t-c(x,y),\]
for a cost function $c(x,y)$. 
The condition ($H_1$)--($H_1^*$) is the counterpart of the condition (A1) assumed on the cost function $c(x,y)$ in \cite{MTW}. 
Note that from \eqref{gs02} and \eqref{gs03}, we can directly derive \eqref{2020} for a dual pair of functions $u,v$.

Similarly, by switching $x$ and $y$, $U$ and $V$, one can also introduce the notion of $\varphi^*$-concavity for the function $v$. 
From Definition \ref{dd11} and \eqref{dgs1}, when $(u,v)\in K$ is a dual pair, $u$ is naturally $\varphi$-concave and $v$ is $\varphi^*$-concave.

Let $u$ be a $\varphi$-concave function in $U$. We define a set-valued mapping $T_{u}=T_{u,\varphi} : U\to V$. 
For any $x_0\in U$, let $T_u(x_0)$ denote the set of points $y_0$ such that $\varphi(x,y_0,s_0)$ is a $\varphi$-support function of $u$ at $x_0$ for some constant $s_0$.
For any subset $E\subset U$, we denote $T_u(E)=\cup_{x\in E}T_u(x)$.

If $u$ is $C^1$ smooth, by condition ($H_1$) $(y_0,s_0)$ is uniquely determined by $(Du(x_0),u(x_0))$, and $T_u$ is single valued. 
In this paper we call the mapping $T_u$ the \emph{$\varphi$-normal mapping} of $u$. 
Similarly we can define the $\varphi^*$-normal mapping for $\varphi^*$-concave functions. In particular, if $(u,v)\in K$ is a dual pair, we see that $y\in T_{u,\varphi}(x)$ if and only if $x\in T_{v,\varphi^*}(y)$. 

\begin{remark}
As the constraint function $\varphi$ is smooth, any $\varphi$-concave function $u$ is semi-concave, namely there exists a constant $C$ such that $u(x)-C|x|^2$ is concave. It follows that $u$ is twice differentiable almost everywhere and $T_u(x)$ is a singleton for almost all $x\in U$. 

\end{remark}

\begin{lemma}\label{lgs1}
Let $(u,v)\in K$ be a dual maximising pair of $I$. Assume that the constraint $\varphi^*$ satisfies condition ($H_1^*$). Let
	\[Y=Y_u=\left\{y\in V\,|\,\exists\, x_1\neq x_2\in U \mbox{ such that } y\in T_u(x_1)\cap T_u(x_2)\right\}.\]
Then $Y$ has Lebesgue measure zero.
\end{lemma}

\begin{proof}
If $y\in T_u(x_1)\cap T_u(x_2)$, we have $x_1, x_2\in T_{v,\varphi^*}(y)$. From the proof of Theorem \ref{t003}, $v$ is almost everywhere differentiable. Assume $y$ is a differentiable point, then by definition
	\begin{eqnarray*}
		\varphi^*(x_i,y,u(x_i))\!\!&=&\!\!-v(y), \\
		\varphi^*_y(x_i,y,u(x_i))\!\!&=&\!\!-Dv(y),
	\end{eqnarray*}
for $i=1,2$. If $x_1\neq x_2$, this is a contradiction to ($H_1^*$).
\end{proof}

We now define a measure $\mu=\mu_{u,g}$ in $U$, where $g\in L^1(V)$ is the positive measurable function in \eqref{ne06}. Set $g\equiv0$ in $\mathbb{R}^n-V$.
For any Borel set $E\subset U$, define
	\begin{equation}\label{gs04}
		\mu(E)=\int_{T_u(E)}g(y)dy.
	\end{equation} 
It follows from Lemma \ref{lgs1} that $\mu$ is a Radon measure, and satisfies the following regularity properties: 
	\[\mu(E)=\inf\{\mu(D)\,:\,E\subset D\subset U,\, D\, \mbox{open}\}\]
for all Borel sets $E\subset U$, and
	\[\mu(D)=\sup\{\mu(K)\,:\,K\subset D,\, K\, \mbox{compact}\}\]
for all open sets $D\subset U$. For further discussion of the measure $\mu$ and its stability property, see \cite{Bak,CY,Fig,Gu,MTW}.

\begin{definition}
A $\varphi$-concave function $u$ is called a generalised solution of \eqref{2pde} if $\mu_{u,g}=fdx$ in the sense of measure, that is for any Borel set $E\subset U$,
	\begin{equation}\label{gs05}
		\int_Ef=\int_{T_u(E)}g.
	\end{equation}
\end{definition}
Note that since we extended $g=0$ to $\mathbb{R}-V$, the boundary condition \eqref{bpde} is a consequence of the mass balance condition \eqref{ne06} in the sense of $|V-T_u(U)|=0$.

The next result shows that a potential function $u$ is a generalised solution. The existence of potential in Theorem \ref{t003} then implies the existence of generalised solutions. But in general we do not have the uniqueness, see Remark \ref{rnl2}.

\begin{theorem}\label{the1}
Let $(u,v)\in K$ be a dual maximising pair of $I$. Then $u$ is a generalised solution of \eqref{2pde}.
\end{theorem}

\begin{proof}
Let $(u,v)\in K$ be a dual maximising pair of $I$. Then by \eqref{s110}, $u$ is $\varphi$-convex and $v$ is $\varphi^*$-convex with respect to each other. By Lemma \ref{p202}, the optimal mapping $T$ associated to $(u,v)$, as determined by \eqref{2014}, is equal to the mapping $T_{u,\varphi}$ almost everywhere on $U$. By Corollary \ref{cc22}, $T$ is measure preserving in the sense of \eqref{jy11}. Hence $u$ is a generalised solution of \eqref{2pde}. Assumption \eqref{ne06} implies that \eqref{bpde} holds. 
\end{proof}


\section{Lagrangian duality}\label{s4}

In this section, we study the dual problem \eqref{2029} of the constrained nonlinear optimisation \eqref{0001new}, and prove Theorem \ref{t004}.
Recall that the Lagrangian function $L$ is defined in \eqref{2024}, where the constraint $\psi$ is given in \eqref{0001new}.
Denote by $I^*$ the optimal value of the primal problem \eqref{0001new}, namely
	\[I^*=\sup_{(u,v)\in K} I(u,v).\]
\begin{definition}\label{d301}
A factor $\mu^*$ is called a \emph{Lagrange multiplier} for the primal problem if $\mu^*\geq0$, and
	\[I^*=\sup\{L(u,v,\mu^*)\,:\,(u,v)\in X\}.\]
\end{definition}

\begin{lemma}\label{la03}
Let $\mu^*$ be a Lagrange multiplier. Then $(u^*,v^*)$ is a global maximum of the primal problem if and only if $(u^*,v^*)$ is feasible and 
	\begin{eqnarray}
		&&(u^*,v^*)=\arg\max_{(u,v)\in X} L(u,v,\mu^*), \label{2025}\\ 
		&&\mu^*\psi(u^*,v^*)=0. \label{2026}
	\end{eqnarray}
\end{lemma}

\begin{proof}
If $(u^*,v^*)$ is a global maximum of the primal problem, then $(u^*,v^*)$ is feasible and furthermore,
	\begin{equation}\label{2027}
	\begin{split}
		I^* &= I(u^*,v^*)\leq I(u^*,v^*)+\mu^*\psi(u^*,v^*) \\
		    &= L(u^*,v^*,\mu^*) \leq \sup\{L(u,v,\mu^*)\,:\,(u,v)\in C(U)\times C(V)\} \\
		    &= I^*,
	\end{split}
	\end{equation}
where the first inequality follows from the definition of Lagrange multiplier ($\mu^*\geq0$) and the feasibility of $(u^*,v^*)$ (i.e. $\psi(u^*,v^*)\geq0$). Using again the definition of Lagrange multiplier, we have $I^*=\sup_{(u,v)\in X}L(u,v,\mu^*)$, so that equality holds throughout \eqref{2027}. This implies the equalities \eqref{2025}--\eqref{2026}.
	
Conversely, if $(u^*,v^*)$ is feasible and \eqref{2025}--\eqref{2026} hold, we have from the definition of Lagrange multiplier,
	\[\begin{split}
	I(u^*,v^*) &= I(u^*,v^*)+\mu^*\psi(u^*,v^*) \\
	 &= L(u^*,v^*,\mu^*)=\max_{(u,v)\in X} L(u,v,\mu^*) = I^*,
	\end{split}\]
so $(u^*,v^*)$ is a global maximum.
\end{proof}

Recall the definitions of the dual functional $J$ in \eqref{2028} and the dual problem in \eqref{2029}.
Note that $J(\mu)$ may be equal to $+\infty$ for some $\mu$. In this case, we define the domain of $J$ to be the set of $\mu$ for which $J(\mu)$ is finite:
	\[D=\{\mu\in\mathbb{R}\,:\,J(\mu)<+\infty\}.\]

Regardless of the functional $I$ and the constraint $\phi$ of the primal problem, the dual problem \eqref{2029} has a nice convexity property, as shown in the following lemma.

\begin{lemma}\label{la04}
The domain $D$ of the dual functional $J$ is convex and $J$ is convex over $D$.
\end{lemma}

\begin{proof}
For any $u,v,\mu,\bar\mu$, and $\alpha\in[0,1]$, we have
	\[L(u,v,\alpha\mu+(1-\alpha)\bar\mu)=\alpha L(u,v,\mu)+(1-\alpha)L(u,v,\bar\mu).\]
Taking the supremum over all $(u,v)\in X$, we obtain
	\[\sup L(u,v,\alpha\mu+(1-\alpha)\bar\mu) \leq \alpha\sup L(u,v,\mu)+(1-\alpha)\sup L(u,v,\bar\mu),\]	
or equivalently
	\[J\left(\alpha\mu+(1-\alpha)\bar\mu\right) \leq \alpha J(\mu)+(1-\alpha)J(\bar\mu).\]
Therefore if $\mu$ and $\bar\mu$ belong to $D$, the same is true for $\alpha\mu+(1-\alpha)\bar\mu$, so $D$ is convex. Furthermore, $J$ is convex over $D$.
\end{proof}

Another important property is that the optimal dual value
	\[J^*=\inf_{\mu\geq0} J(\mu)\]
is always an upper bound of the optimal primal value, as shown in the next lemma.

\begin{lemma}\label{p203}
We have 
	\[I^*\leq J^*.\]
\end{lemma}

\begin{proof}
For all $\mu\geq0$, and $(u,v)\in X$ with $\psi(u,v)\geq0$, we have
	\[\begin{split}
		J(\mu)&=\sup_{(\tilde u,\tilde v)\in X} L(\tilde u,\tilde v,\mu) \\
		 &\geq I(u,v)+\mu \psi(u,v) \geq I(u,v),
	\end{split}\]
and therefore,
	\[J^*=\inf_{\mu\geq0} J(\mu)\geq\sup_{(u,v)\in K} I(u,v)=I^*.\]	
\end{proof}

In the language of nonlinear programming \cite{Ber,GY}, if $J^*=I^*$ we say that \emph{there is no duality gap}; if $J^*>I^*$ \emph{there is duality gap}.
Note that if there exists a Lagrange multiplier $\mu^*$, the above lemma ($J^*\geq I^*$) and the definition of Lagrange multiplier ($I^*=J(\mu^*)\geq J^*$) imply that there is no duality gap.

The following is a sufficient condition for the existence of Lagrange multiplier, which is also a proof of Theorem \ref{t004}.
\begin{lemma}\label{la06}
Let the assumptions \eqref{2030}--\eqref{2032} hold for the primal problem \eqref{0001}. 
Then there is no duality gap and there exists at least one Lagrange multiplier.
\end{lemma}

\begin{proof}
Consider the subset of $\mathbb{R}^2$ given by
	\[\begin{split}
		A=\{(z,w)\,:\ \ &\exists\, (u,v)\in X \mbox{ such that }\\ 
			& \psi(u,v)\geq z,\quad I(u,v)\geq w\}.
	\end{split}\]
We first show that $A$ is convex. Let $(z,w)\in A$ and $(\tilde z,\tilde w)\in A$ be two different elements, we show that their convex combinations belong to $A$. 

The definition of $A$ implies that for some $(u,v)\in X$ and $(\tilde u,\tilde v)\in X$, we have
	\begin{eqnarray*}
		I(u,v)\geq w,\quad \psi(u,v)\geq z, \\
		I(\tilde u,\tilde v)\geq\tilde w,\quad \psi(\tilde u,\tilde v)\geq \tilde z.
	\end{eqnarray*}
For any $\alpha\in[0,1]$, by the concavity of $F$ in \eqref{2030}, we obtain
	\[\begin{split}
		I\left(\alpha u+(1-\alpha)\tilde u,\alpha v+(1-\alpha)\tilde v\right) &\geq \alpha I(u,v)+(1-\alpha)I(\tilde u,\tilde v) \\
		&\geq \alpha w+(1-\alpha)\tilde w. 
	\end{split}\]
By the convexity of $\varphi$ in \eqref{2031} and noting that $\inf_\Omega(f+h)\geq\inf_\Omega f+\inf_\Omega h$ for any $f,h\in C^0(\Omega)$, we obtain
	\[\begin{split}
		\psi\left(\alpha u+(1-\alpha)\tilde u,\alpha v+(1-\alpha)\tilde v\right) &\geq \alpha \psi(u,v)+(1-\alpha)\psi(\tilde u,\tilde v) \\
		&\geq \alpha z+(1-\alpha)\tilde z.
	\end{split}\]
Since the combination $\left(\alpha u+(1-\alpha)\tilde u,\alpha v+(1-\alpha)\tilde v\right)\in X$, the above equations imply that the convex combination of $(z,w)$ and $(\tilde z,\tilde w)$, i.e. $$\left(\alpha z+(1-\alpha)\tilde z,\alpha w+(1-\alpha)\tilde w\right),$$ belongs to $A$. This proves the convexity of $A$.

We next observe that $(0,I^*)$ is not an interior point of $A$; otherwise, the point $(0,I^*+\varepsilon)$ would belong to $A$ for some small $\varepsilon>0$, contradicting the definition of $I^*$ as the optimal primal value. 

Therefore, there exists a supporting hyperplane passing through $(0,I^*)$ and containing $A$ in one side. In particular, there exists a vector $(\mu,\beta)\neq(0,0)$ such that
	\begin{equation}\label{2033}
		\beta I^*\geq\mu z+\beta w,\qquad\forall\, (z,w)\in A.
	\end{equation}
We observe that if $(z,w)\in A$, then $(z,w-\gamma)\in A$ and $(z-\gamma,w)\in A$ for all $\gamma>0$. The inequality \eqref{2033} thus implies that 
	\begin{equation}\label{2034}
		\mu\geq0,\qquad \beta\geq0.
	\end{equation}

We now claim that $\beta>0$. If not, $\beta=0$ and from \eqref{2033} 
	\[0\geq\mu z,\qquad\forall\, (z,w)\in A.\]
By the assumption \eqref{2032}, there exists a pair $(\bar u,\bar v)\in X$ such that
	\[\psi(\bar u,\bar v)>0.\]
Since $(\psi(\bar u,\bar v),I(\bar u,\bar v))\in A$, we have
	\[0\geq\mu \psi(\bar u,\bar v),\]
which in view of $\mu\geq0$ in \eqref{2034} implies that $\mu=0$. This means, however, that $(\mu,\beta)=(0,0)$ arriving at a contradiction. 
Thus, we must have $\beta>0$ and by dividing if necessary the vector $(\mu,\beta)$ by $\beta$, we may assume that $\beta=1$. Note that 
	\[\left(\psi(u,v),I(u,v)\right)\in A,\qquad\forall\, (u,v)\in X.\]
Equation \eqref{2033} implies that
	\begin{equation}
		I^*\geq I(u,v)+\mu \psi(u,v),\qquad\forall\, (u,v)\in X.
	\end{equation}
	
Taking the supremum over $(u,v)\in X$ and using the fact $\mu\geq0$, we obtain
	\[\begin{split}
		I^*&\geq\sup_{(u,v)\in X} L(u,v,\mu) \\
		 &=J(\mu)\geq J^*,
	\end{split}\]
where $J^*$ is the optimal dual value. By Lemma \ref{p203} we have the equalities hold above, namely $\mu$ is a Lagrange multiplier and there is no duality gap.
\end{proof}

\section{Examples and applications}\label{s5}

In this section we present some interesting examples and applications of the nonlinear optimisation \eqref{prob}.

\subsection{Optimal transportation}\label{ss51}

Let $U,V$ be two bounded domains in $\mathbb{R}^n$, and $c\in C^4(U\times V)$ be a cost function. 
Let $f,g$ be two positive densities supported on $U,V$, respectively, satisfying the mass balance condition
  \begin{equation}\label{otmb}
  	\int_Uf=\int_{V}g.
  \end{equation}
The optimal transport problem is to find a measure preserving mapping $T_0 : U\to V$ minimising the cost functional
  \[\mathcal{C}(T)=\int_U c(x,T(x))f(x)dx\]
among all measure preserving mappings $T :U\to V$, (see Definition \ref{mpmap}).
Denote the set of measure preserving mappings by $\mathcal{T}$.

Kantorovich introduced a dual functional
  \begin{equation}\label{ex01}
    I(u,v)=\int_U u(x)f(x)dx+\int_{V}v(y)g(y)dy
  \end{equation}                                                                               
over the set
  \begin{equation}\label{ex02}
    K=\{(u,v)\in C(U)\times C(V)\,:\,u(x)+v(y)\leq c(x,y),\quad\forall\, x\in U, y\in V\}.
  \end{equation}
Under suitable conditions, one can prove that
  \[\inf_{T\in\mathcal{T}}\mathcal{C}(T)=\sup_{(u,v)\in K}I(u,v).\]
The reader is referred to \cite{Am,Caf,E,GM1,MTW,U,V} for further discussion on the optimal transport problem. 

Note that \eqref{ex01}--\eqref{ex02} is a linear case of \eqref{0001}--\eqref{prob}. In the form of \eqref{0001new}, one can set
	\begin{eqnarray*}
		F(x,y,u,v) \!\!&=&\!\! u(x)f(x)+v(y)g(y),\quad d\gamma=dx\otimes dy,\\
		\phi(x,y,u,v) \!\!&=&\!\! u(x)+v(y)-c(x,y),
	\end{eqnarray*}
or equivalently, $\varphi(x,y,v)=v-c(x,y)$ in \eqref{ne01} and $\varphi^*(x,y,u)=u-c(x,y)$ in \eqref{gs02}.
All the hypotheses in Theorem \ref{t003} are satisfied when the cost function $c(x,y)$ satisfies the following conditions \cite{MTW}:
	\begin{itemize}
		\item[(A1)] For any $x,p\in\mathbb{R}^n$, there is a unique $y\in\mathbb{R}^n$ such that $D_xc(x,y)=p$; and for any $y,q\in\mathbb{R}^n$, there is a unique $x\in\mathbb{R}^n$ such that $D_yc(x,y)=q$.
	\end{itemize}
The hypotheses in Theorem \ref{t003} follow from the constructions of $F,\phi$ together with $f>0,g>0$. The mass balance condition \eqref{otmb} implies \eqref{ne06}. The condition (A1) implies both ($H_1$) and ($H_1^*$).

Therefore, by Theorem \ref{t003} we have the existence of potentials $(u,v)$ and optimal mapping $T$. 
The existence of potentials in optimal transportation was previously proved in \cite{Bre,Caf,GM1}.
By directly applying the formula \eqref{2pde}, one obtains the optimal transportation equation
  \begin{equation}\label{ote}
   \left|\det\,[D^2u-D^2_{xx}c]\right|=|\det\,c_{x,y}|\frac{f}{g}.
  \end{equation}

We remark that in the linear case \eqref{ex01}--\eqref{ex02}, both $F$ in \eqref{0001} and $\phi$ in \eqref{0003} are linear in $t,s$ variables, that is a border situation of $F$ being concave and $\phi$ being convex in $t,s$, simultaneously. 

There are numerous applicatons of the optimal transportation. Here we mention two important ones in geometric optics. In \cite{W04}, Xu-Jia Wang showed that the far field reflector problem is an optimal transport problem, and so is a linear optimisation problem. The associated cost function $c(x,y)=-\log(1-x\cdot y)$, where $x, y$ are points on the unit sphere $\mathbb{S}^2$. 
Later on in \cite{GH1} Guti\'errez and Huang showed that the far field refractor problem is also an optimal transport problem. Let $\kappa$ be the refractor index from the initial media to the target media. Then the associated cost function $c(x,y)=-\log(1-\kappa x\cdot y)$ when $\kappa<1$; and $c(x,y)=\log(\kappa x\cdot y-1)$ when $\kappa>1$, where $x,y$ are points on the unit sphere $\mathbb{S}^n$. 
In the following subsections, we will consider more general (near field) reflector and refractor problems and show that they are in the class of the nonlinear optimisation \eqref{0001}--\eqref{prob}.

\subsection{Near field reflector problem with point source}\label{ss52}

In \cite{Liu} it is proved that the near field reflector problem is a nonlienar optimisation. For the convenience of the reader, we summaries the arguments as follows.  
Assume that the light emits from the origin $O$ and passes through $\Omega\subset\mathbb{S}^n$ with a positive density $f\in L^1(\Omega)$. 
After being reflected from a surface $\Gamma$, the light will illuminate the target surface $\Omega^*$ in $\mathbb{R}^{n+1}$ with a prescribed positive density $g\in L^1(\Omega^*)$. 
Assume the energy conservation condition
  \begin{equation}\label{a001}
   \int_\Omega f=\int_{\Omega^*}g.
  \end{equation}

Represent the reflector $\Gamma$ in polar coordinate system as
	\[\Gamma_\rho=\{X\rho(X)\,:\,X\in\Omega)\},\]
where $\rho$ is a positive function. 
Recall that \cite{KW}, $\Gamma_\rho$ is \emph{admissible} if at each point $X\rho(X)\in\Gamma$ there exists a \emph{supporting ellipsoid}. Therefore, the radial function $\rho$ satisfies
	\begin{equation}\label{419a}
		\rho(X)=\inf_{Y\in\Omega^*}\frac{p(Y)}{1-\epsilon(p(Y))\langle X,\frac{Y}{|Y|}\rangle},\quad X\in\Omega,
	\end{equation}
where $p$ is the focal function on $\Omega^*$ and $\epsilon(p)=\sqrt{1+p^2/|Y|^2}-p/|Y|$ is the eccentricity.
Because there is an ellipsoid $E_{Y,p(Y)}$ supporting to $\Gamma_\rho$ for each $Y\in\Omega^*$, we also have
	\begin{equation}\label{420a}
		p(Y)=\sup_{X\in\Omega}\rho(X)\left[1-\epsilon(p(Y))\langle X,\frac{Y}{|Y|}\rangle\right],\quad Y\in\Omega^*.
	\end{equation}
Note that in \eqref{419a} for each $X\in\Omega$ the infimum is achieved at some $Y\in\Omega^*$ and in \eqref{420a} for each $Y\in\Omega^*$ the supremum is achieved at some $X\in\Omega$.

The relations \eqref{419a}--\eqref{420a} between the radial and focal functions of a reflector $\Gamma$ are analogous to the classical relations between the radial and support functions for convex bodies, for example, see \cite{Sch}. Inspired by that and \cite{W04}, we set $\eta=1/p$. Then the pair $(\rho,\eta)$ satisfies the dual relation
	\begin{eqnarray}\label{421a}
		\rho(X)\!\!&=&\!\!\inf_{Y\in\Omega^*}\frac{1}{\eta(Y)\left(1-\epsilon(\eta(Y))\langle X,\frac{Y}{|Y|}\rangle\right)}, \\
		\eta(Y)\!\!&=&\!\!\inf_{X\in\Omega}\frac{1}{\rho(X)\left(1-\epsilon(\eta(Y))\langle X,\frac{Y}{|Y|}\rangle\right)}. \nonumber
	\end{eqnarray}

Similarly to \cite{W04}, we can now formulate the reflector problem to a nonlinear optimisation \eqref{0001}--\eqref{prob} as follows. 
Let $u=\log\rho$ and $v=\log\eta$. Set the functional
	\begin{equation}\label{2029b}
		I(u,v)=\int_{\Omega}f(X)u + \int_{\Omega^*}g(Y)\left(v+\log (1-\frac{\langle T^{-1}Y,Y\rangle}{e^{-v}+\sqrt{|Y|^2+e^{-2v}}})\right),
	\end{equation}
and the constraint set
	\[K=\left\{(u,v)\in C(\Omega)\times C(\Omega^*)\,:\,\phi(X,Y,u,v)\leq0\right\},\]
with the constraint function
	\begin{equation}\label{2030b}
		\phi(X,Y,u,v)=u+v+\log\left(1-\frac{\langle X,Y\rangle}{e^{-v}+\sqrt{|Y|^2+e^{-2v}}}\right).
	\end{equation}

We assume a further condition on domains $\Omega$ and $\Omega^*$: $\Omega^*$ is contained in the cone $\mathcal{C}_V=\{tX\,:\,t>0,X\in V\}$ for a domain $V\subset\mathbb{S}^n$ and
	\begin{equation}\label{ah01}
		\overline{\Omega}\cap\overline{V}=\emptyset,
	\end{equation}
where $\overline{\Omega}$ and $\overline{V}$ denote the closure of $\Omega$ and $V$. It implies that there exists a small constant $\delta_0>0$, such that for any $X\in\Omega, Y\in\Omega^*$,
	\begin{equation}\label{rp01}
		-1\leq\langle X,\frac{Y}{|Y|}\rangle\leq1-\delta_0.
	\end{equation}
Under the assumption \eqref{rp01}, one can verify that \eqref{2004} is satisfied by \eqref{2029b} and \eqref{2030b}.
The energy conservation condition \eqref{a001} implies \eqref{ne06}.
To show the condition ($H_1$), let $(u,v)\in K$ be a dual maximising pair of $I$, (see Theorem \ref{t003}). If there holds
	\[\varphi_x(X,Y,v(Y))=-Du(X),\quad \varphi(X,Y,v(Y))=-u(X),\]
at $X\in\Omega, Y\in\Omega^*$, it was proved \cite{Liu} that $Y$ is the target point of light emitting along $X$, reflected at $Xe^{u(X)}\in\Gamma$ with unit normal
	\[\gamma=\frac{(Du,0)-(1+Du\cdot x)X}{\sqrt{1+|Du|^2-(Du\cdot x)^2}}.\]
By the reflection law, $Y_r=X-2\langle X,\gamma\rangle\gamma$,
	\[Y=Xe^{u(X)}+Y_r|Y-Xe^{u(X)}|\]
is uniquely determined. It implies that $(\varphi_x,\varphi)(X,\cdot,\cdot)$ is one-to-one in $\Omega^*\times v(\Omega^*)$ for each $X\in\Omega$, see Remark \ref{rrr1}.

As a consequence of Theorem \ref{t003}, we have \cite{Liu}
\begin{corollary}\label{t101}
Assume that $f, g$ satisfy \eqref{a001}. Suppose that $\Omega$ and $\Omega^*$ satisfy \eqref{ah01}.
Then there is a dual maximising pair $(u,v)\in K$ satisfing
	\begin{equation*}
		I(u,v)=\sup_{(u,v)\in K}I(u,v),
	\end{equation*} 
where $I(u,v)$ is in \eqref{2029b}, and the constraint $\phi$ is in \eqref{2030b}.
Moreover, $\rho=e^u$ is a solution of the reflector problem with given densities $(\Omega, f)$ and $(\Omega^*, g)$.
(Note that the solutions need to be understood as generalised solutions.)
\end{corollary}

By directly applying the formula \eqref{2pde}, we also obtain the PDE in the near field case \cite{Liu}, which was previously obtained by Karakhanyan and Wang in \cite{KW}.
Assume that $\Omega^*$ is given implicitly by
	\begin{equation}\label{a005}
		\Omega^*=\{Z\in\mathbb{R}^{n+1}\,:\,\psi(Z)=0\}.
	\end{equation}
Suppose that $\Omega$ is a subset of upper unit sphere $\mathbb{S}^n_+=\mathbb{S}^n\cap\{x_{n+1}>0\}$. Let $X=(x,x_{n+1})$ be a parameterisation of $\Omega$, where $x_{n+1}=\sqrt{1-|x|^2}=:\omega(x)$, and $x=(x_1,\cdots,x_n)$. For simplification, we define some auxiliary functions
	\begin{eqnarray}
		a \!\!&=&\!\! |D\rho|^2-(\rho+D\rho\cdot x)^2, \label{a006}\\
		b \!\!&=&\!\! |D\rho|^2+\rho^2-(D\rho\cdot x)^2, \label{a007}\\
		t \!\!&=&\!\! \frac{\rho x_{n+1}-y_{n+1}}{\rho x_{n+1}},\quad \beta=\frac{t}{(Y-X\rho)\cdot\nabla\psi}, \label{a008}
	\end{eqnarray}
and denote the matrix
	\begin{equation}\label{a009}
		\mathcal{N}=\{\mathcal{N}_{ij}\},\quad \mathcal{N}_{ij}=\delta_{ij}+\frac{x_ix_j}{1-|x|^2}.
	\end{equation}
By computing in the local orthonormal frame, we obtain the equation as follows	
\begin{corollary}\label{t102}
The function $\rho$ is a solution of
	\begin{equation}\label{a010}
		\left|\det\,\left[D^2\rho-\frac{2}{\rho}D\rho\otimes D\rho-\frac{a(1-t)}{2t\rho}\mathcal{N}\right]\right|=\left|\frac{a^{n+1}}{t^nb\beta}\right|\frac{f}{2^n\rho^{2n+1}\omega^2g|\nabla\psi|}.
	\end{equation}
\end{corollary}
Note that the matrix in equation \eqref{a010} has a different sign to that in \cite{KW}, since we calculate the absolute value of the determinant.

\subsection{Near field reflector problem with parallel source}\label{ss53}

The ideal reflection system can be described as follows: a parallel light emits from $\Omega\subset\mathbb{R}^n\times\{0\}$ along $e_{n+1}=(0,\cdots,0,1)$ with a positive density $f\in L^1(\Omega)$. After being reflected by the surface $\Gamma$ in $\mathbb{R}^{n+1}$, it will illuminate the target domain $\Omega^*\subset\mathbb{R}^n\times\{0\}$ with the prescribed density $g\in L^1(\Omega^*)$. 
Assume the energy conservation condition
	\begin{equation}\label{rpp1}
		\int_\Omega f=\int_{\Omega^*}g.
	\end{equation}
We represent the reflector $\Gamma$ as graph $u|_\Omega$, namely
	\[\Gamma_u=\{(x,u(x))\,:\,x\in\Omega\},\]
where $u$ is a positive function.

Consider the \emph{``inverse" paraboloid} with focus at $y\in\Omega^*$ and the axial direction $-e_{n+1}$.
The reflection property of these paraboloids is that: all the incident light from parallel source along the direction $e_{n+1}$ will be reflected by the inverse paraboloids to the focus points $y$.
Such an inverse paraboloid can be represented by $\Gamma_p=\{(x,p(x))\,:\,x\in\mathbb{R}^n\}$ with
	\begin{equation}\label{2060}
		p(x)=p_{y}(x)=\frac{1}{2v}-\frac{v}{2}|x-y|^2,
	\end{equation}
where $v>0$ is a constant depending on $y$ such that $p(x)>0$ in $\Omega$. 

If we regard $v=v(y)$ as a function on $\Omega^*$, we then have a family of paraboloids $\{\Gamma_{p_y}\,:\,y\in\Omega^*\}$. It is natural to construct the reflector $\Gamma_u$ by the envelope of the family of paraboloids $\{\Gamma_{p_y}\,:\,y\in\Omega^*\}$. 

In an ideal system, at each point $(x,u(x))\in\Gamma_u$ on an \emph{admissible} reflector $\Gamma_u$ there exists a \emph{supporting paraboloid}, namely for some $y\in\Omega^*$
	\[u(x)=\frac{1}{2v(y)}-\frac{v(y)}{2}|x-y|^2,\quad\mbox{and}\quad u(x')\leq\frac{1}{2v(y)}-\frac{v(y)}{2}|x'-y|^2\ \ \forall\, x'\in\Omega.\]
Hence, the defining function $u$ and the dual function $v$ satisfy the dual relation:
	\begin{eqnarray}
		u(x)\!\!&=&\!\!\inf_{y\in\Omega^*}\left\{\frac{1}{2v(y)}-\frac{v(y)}{2}|x-y|^2\right\},\quad x\in\Omega, \label{2061}\\
		v(y)\!\!&=&\!\!\inf_{x\in\Omega}\left\{\frac{-u+\sqrt{u^2+|x-y|^2}}{|x-y|^2}\right\},\quad y\in\Omega^*. \label{2062}
	\end{eqnarray}
Note that in \eqref{2061}, if for $x\in\Omega$ the infimum is achieved at some $y=T(x)\in\Omega^*$, then in \eqref{2062} at $y=T(x)$ the infimum is achieved at $x\in\Omega$. By \eqref{2061} one has
	\[Du(x)=-v(x-y),\]
and thus
	\[|Du(x)|^2=1-2vu.\]
Since $u>0,v>0$, we have the natural restriction $|Du|<1$ on $\Omega$.	

Based on \eqref{2061}--\eqref{2062} we can now formulate this problem to a nonlinear optimisation problem \eqref{0001}--\eqref{prob}.
Set the functional 
	\begin{equation}\label{2063}
		I(u,v)=\int_{\Omega} f(x)u(x) + \int_{\Omega^*}g(y)\left(\frac{v}{2}|T^{-1}y-y|^2-\frac{1}{2v}\right),
	\end{equation}
and the constraint set
	\[\mathcal{K}=\{(u,v)\in C(\Omega)\times C(\Omega^*)\,:\,\phi(x,y,u,v)\leq0\},\]
with the constraint function
	\begin{equation}\label{2064}
		\phi(x,y,u,v)=u-\frac{1}{2v}+\frac{v}{2}|x-y|^2.
	\end{equation}
Note that $\phi$ can be written as \eqref{ne01} with $\varphi(x,y,v)=-\frac{1}{2v}+\frac{v}{2}|x-y|^2$. By differentiating, we have
	\begin{eqnarray}
		&&\varphi_x=(x-y)v,\quad \varphi_y=-(x-y)v,\quad \varphi_{xx}=vI,\quad \varphi_{xy}=-vI,  \label{2065}\\
		&&\varphi_s=\frac{1}{2v^2}+\frac{1}{2}|x-y|^2,\quad \varphi_{xs}=x-y,\quad etc. \nonumber
	\end{eqnarray}
From Theorem \ref{t003} there exists a dual maximising pair $(u,v)\in K$ of $I$. From \eqref{2065} one can see that $(\varphi_x,\varphi)(x,\cdot,\cdot)$ is one-to-one in $\Omega^*\times v(\Omega^*)$ for each $x\in\Omega$, Remark \ref{rrr1}. The energy conservation condition \eqref{rpp1} implies \eqref{ne06}.
Alternatively, one can directly verify that the optimal mapping $T$ associated to $(u,v)$, determined by \eqref{2014}, is equal to the reflection mapping. This implies the condition ($H_1$) due to the reflection law. 

\begin{proposition}\label{p222}
Let $(u,v)\in K$ be a dual maximising pair of $I$, $|Du|<1$.
The associated optimal mapping $T$ determined by \eqref{2014} is equal to the reflection mapping $T_r$ obtained by the reflection law.
\end{proposition}

\begin{proof}
From \eqref{2014} and \eqref{2065}, at $(x,y)=(x,T(x))$ we have
	\begin{equation}\label{2067}
		x-y=-\frac{Du}{v}.	
	\end{equation}
From \eqref{as90new}, $u-\frac{1}{2v}+\frac{v}{2}|x-y|^2=0$, thus 
	\begin{equation}\label{2068}
		v=\frac{1-|Du|^2}{2u}.
	\end{equation}
Note that $u>0, v>0$ and $|Du|<1$. Combining \eqref{2067} and \eqref{2068}, we obtain
	\begin{equation}\label{2069}
		x-y=-\frac{2uDu}{1-|Du|^2}.
	\end{equation}
This implies the optimal mapping $T$ is given by
	\[T(x)=x+\frac{2uDu}{1-|Du|^2}.\]
	
Next, we calculate the reflection mapping $T_r$.
At point $(x,u(x))\in\Gamma_u$, from direct calculations the unit normal is
	\[\gamma=\frac{(Du,-1)}{\sqrt{1+|Du|^2}}.\]
By the reflection law, we have the reflected direction
	\begin{equation}\label{2801}
		Y_r=e_{n+1}-2\langle e_{n+1},\gamma\rangle\gamma=\frac{(2Du,|Du|^2-1)}{|Du|^2+1}.
	\end{equation}
On the other hand, since the reflected ray meets the hyperplane $\mathbb{R}^n\times\{0\}$, we have
	\begin{equation}\label{2802}
		Y_r=\frac{(y-x,-u)}{\sqrt{|y-x|^2+u^2}}.
	\end{equation}
From \eqref{2801}--\eqref{2802} we obtain that
	\begin{equation}\label{2803}
		y=T_r(x)=x+\frac{2uDu}{1-|Du|^2}=T(x).
	\end{equation}
\end{proof}

Therefore, as a consequence of Theorem \ref{t003}, we have
\begin{corollary}
Assume that $f,g$ satisfy the energy conservation condition \eqref{rpp1}. 
Then there is a dual maximising pair $(u,v)\in K$ of $I$, where $I, K$ are in \eqref{2063}, \eqref{2064}. 
Moreover, $u$ is a (generalised) solution of the reflector problem with given densities $(\Omega,f)$ and $(\Omega^*,g)$.
\end{corollary}

We now derive the equation for the potential function $u$ by using the formula \eqref{2pde},	
	\begin{equation}\label{2066}
	\begin{split}
		\left|\det\,\left[D^2u+vI\right]\right|&=v^n\left|\det\,\left[I-\frac{1}{\varphi_s}(x-y)\otimes(x-y)\right]\right|\frac{f}{g}, \\
			&=v^n\left|1-\frac{1}{\varphi_s}|x-y|^2\right|\frac{f}{g},
	\end{split}
	\end{equation}
where in the second step we used the formula $\det\,[I+\xi\otimes\eta]=1+\xi\cdot\eta$ for any vectors $\xi, \eta\in\mathbb{R}^n$.
Combining \eqref{2065}--\eqref{2069}, we obtain that
	\begin{equation}\label{2070}
		\phi_s=\frac12\frac{|x-y|^2}{|Du|^2}+\frac12|x-y|^2=\frac{1+|Du|^2}{2|Du|^2}|x-y|^2.
	\end{equation}
Therefore, the equation \eqref{2066} becomes
	\begin{equation}\label{2071}
		\det\,\left[D^2u+\frac{1-|Du|^2}{2u}I\right]=\frac{(1-|Du|^2)^{n+1}}{(2u)^n(1+|Du|^2)}\frac{f}{g}.
	\end{equation}
	
The equation \eqref{2071} was also obtained in \cite{LT} by directly computing the Jacobian of the reflection mapping. 
Denote the matrix
	\[A(u,Du)=\frac{1-|Du|^2}{2u}I.\]
It satisfies the (A3) condition in \cite{MTW} without the orthogonal restriction, provided that $u>0$. The regularity of \eqref{2071} follows from \cite{LT,MTW}. Indeed, the considerations in \cite{MTW} stemmed from the treatment of the reflector antenna problem by Wang in \cite{W96}, which can be represented as an optimal transport problem on the sphere $\mathbb{S}^n$ with the cost function $c(x,y)=-\log(1-x\cdot y)$.

\subsection{Near field refractor problem with point source}\label{ss54}

The near field refractor problem has been studied by Guti\'errez and Huang \cite{GH}. 
Suppose the light emits from the origin surrounded by medium $I$ with positive intensity $f(X)$ for $X\in\Omega$, where $\Omega\subset\mathbb{S}^n$. There is a surface $\mathcal{R}$, separates two homogeneous and isotropic media $I$ and $I\!I$, such that all rays refracted by $\mathcal{R}$ into medium $I\!I$ illuminate a target hypersurface $\Omega^*$ in $\mathbb{R}^{n+1}$ with positive intensity $g$ on $\Omega^*$. Assume that $f,g$ satisfy the energy conservation condition
	\begin{equation}\label{nl01}
		\int_\Omega f=\int_{\Omega^*}g.
	\end{equation}

Let $n_1, n_2$ be the indices of refraction of media $I$, $I\!I$, respectively, and $\kappa=n_2/n_1$. When $\kappa<1$, the refracted rays tend to bent away from the normal, when $\kappa>1$, the refracted rays tend to bent towards the normal. 

There is a special interface surface $\mathcal{S}$ between media $I$ and $I\!I$, called Cartesian oval \cite{GH}, that refracts all rays emitting from the origin $O$ into the point $Y$. In the polar coordinates, represent 
	\[\mathcal{S}=\{X\rho_o(X)\,:\,X\in\mathbb{S}^n\}.\]
In the case $\kappa<1$, by the Snell law of refraction one has
	\begin{equation}\label{nl02}
		\rho_o(X)=h(X,Y,p)=\frac{(p-\kappa^2X\cdot Y)-\sqrt{\Delta(X\cdot Y)}}{1-\kappa^2},
	\end{equation}
where $p$ is the focal parameter, $\kappa|Y|<p<|Y|$ and $\Delta(t)=(p-\kappa^2t)^2-(1-\kappa^2)(p^2-\kappa^2|Y|^2)$ for $t\in\mathbb{R}$. For non-degenerate Cartesian ovals, there are some physical constraints for refraction
  \begin{equation}\label{nl08}
   \kappa|Y|<p<|Y|,\quad \rho_o\leq p,\quad \mbox{ and } p\leq X\cdot Y,\quad \forall\, X\in\Omega.
  \end{equation}
We refer the readers to \cite{GH} for more physical interpretations and detailed calculations. 

If we regard $p=p(Y)$ as a focal function on $\Omega^*$, we then have a family of Cartesian ovals. 
Represent the surface $\mathcal{R}$ in polar coordinate system as
	\[\mathcal{R}_\rho=\{X\rho(X)\,:\,X\in\Omega\},\]
where $\rho$ is a positive function. Recall that \cite{GH}, $\mathcal{R}$ is a near field refractor if at each point $X\rho(X)\in\mathcal{R}$ there exists a supporting Cartesian oval, i.e., for some $Y\in\Omega^*$, $\rho(X')\leq\rho_o(X',Y,p(Y))$ for all $X'\in\Omega$ with equality holds at $X'=X$. Therefore, the radial function $\rho$ satisfies 
	\begin{equation}\label{nl03}
		\rho(X)=\inf_{Y\in\Omega^*}\frac{(p-\kappa^2X\cdot Y)-\sqrt{\Delta(X\cdot Y)}}{1-\kappa^2},\quad X\in\Omega.
	\end{equation}
From the energy conservation, for each $Y\in\Omega^*$ there is an oval $\rho_0(\cdot,Y,p(Y))$ supporting to $\mathcal{R}_\rho$. We also have
	\begin{equation}\label{nl04}
		p(Y)=\sup_{X\in\Omega}(1-\kappa^2)\rho(X)+\kappa^2X\cdot Y+\sqrt{\Delta(X\cdot Y)},\quad Y\in\Omega^*.
	\end{equation}
The above relations are analogous to \eqref{419a}--\eqref{420a}. By setting $\eta=1/p$, the pair $(\rho,\eta)$ satisfies the dual relation
	\begin{eqnarray}\label{n105}
		&&\rho(X)=\inf_{Y\in\Omega^*}\frac{1-\eta(\kappa^2X\cdot Y+\sqrt{\Delta(X\cdot Y)})}{\eta(1-\kappa^2)}, \\
		&&\eta(Y)=\inf_{X\in\Omega}\frac{1-\eta(\kappa^2X\cdot Y+\sqrt{\Delta(X\cdot Y)})}{\rho(1-\kappa^2)}.
	\end{eqnarray}
Similarly to \eqref{2029b} we now formulate the refractor problem to the following nonlinear optimisation, which is more complicated than \eqref{2029b}. Let $u=\log\rho$ and $v=\log\eta$. Set the functional
	\begin{equation}\label{nl06}
		I(u,v) = \int_{\Omega}f(X)u + \int_{\Omega^*} g(Y)\left(v+\log(\frac{1-\kappa^2}{1-e^v(\kappa^2T^{-1}Y\cdot Y+\sqrt{\Delta(T^{-1}Y\cdot Y)}) } )\right),
	\end{equation}
and the constraint set 
	\[K=\{(u,v)\in C(\Omega)\times C(\Omega^*)\,:\,\phi(X,Y,u,v)\leq0\},\]
with the constraint function
	\begin{equation}\label{nl07}
		\phi(X,Y,u,v)=u+v+\log\left(\frac{1-\kappa^2}{1-e^v(\kappa^2X\cdot Y+\sqrt{\Delta(X\cdot Y)})}\right).
	\end{equation}

As in \cite{GH}, we make the following assumptions on $\Omega$ and $\Omega^*$, which are due to the physical constraints for refraction \eqref{nl08}:
\begin{itemize}
 \item[(R1)] There exists $\tau$ with $0<\tau<1-\kappa$ such that $X\cdot Y\geq(\kappa+\tau)|Y|$ for all $x\in\overline\Omega$ and all $Y\in\overline\Omega^*$.
 \item[(R2)] Let $0<r_0\leq\frac{\tau}{1+\kappa}\dist(0,\overline\Omega^*)$ and consider the cone in $\mathbb{R}^{n+1}$
  \[Q_{r_0}=\{tX\,:\,X\in\overline\Omega, 0<t<r_0\}.\]
 For each $\xi\in\mathbb{R}^n$ and for each $X\in Q_{r_0}$ we assume that $\overline\Omega^*\cap\{X+t\xi\,:\,t\geq0\}$ contains at most one point. That is, for each $X\in Q_{r_0}$ each ray emanating from $X$ intersects $\overline\Omega^*$ at most in one point. 
\end{itemize}

Note that from \eqref{nl02}
  \begin{equation}\label{nl09}
   \Delta(X\cdot Y)=\kappa^2p^2-2\kappa^2(X\cdot Y)p+\kappa^2(X\cdot Y)^2+\kappa^2(1-\kappa^2)|Y|^2.
  \end{equation}
Since $p\leq X\cdot Y$ by \eqref{nl08}, $\Delta(X\cdot Y)$ is decreasing in $p$. Since $p=e^{-v}$, $\Delta(X\cdot Y)$ is increasing in $v$. Hence \eqref{nl06}--\eqref{nl07} satisfy the assumption \eqref{2004}, where the constant in \eqref{2004} depends on $\kappa, \Omega, \Omega^*$ and the assumptions (R1) and (R2). See Remark 4.2 in \cite{GH} for more physical interpretations of (R1) and (R2). 

By the proof of Theorem \ref{t003} and Remark \ref{rnl2} we have the following existence result in the near field refractor problem, which was previously obtained in \cite{GH}. Note that the solutions need to be understood as generalised solutions.  
\begin{corollary}
Assume that $f,g$ satisfy \eqref{nl01} and (R1)--(R2). Then given $P_0\in Q_{r_0}$ with $0<|P_0|\leq\left(\frac{1-\kappa}{1+\kappa}\right)^2r_0$, there exists a dual maximising pair $(u,v)\in K$ such that $u(P_0/|P_0|)=\log|P_0|$. Moreover, $\rho=e^u$ is a solution of the refractor problem such that $\mathcal{R}_\rho$ passes through the point $P_0$.
\end{corollary}

\begin{remark}
In the case $\kappa>1$, by the physical constraints for $|Y|<p<\kappa|Y|$ the refracting piece of Cartesian oval is given by
  \[\mathcal{O}(Y,p)=\left\{\rho_0(X,Y,p)\,:\,X\cdot Y\geq\frac{p+\sqrt{(\kappa^2-1)(\kappa^2|Y|^2-p^2)}}{\kappa^2}\right\}\]
with
  \begin{equation}\label{nl10}
   \rho_0(X,Y,p)=\frac{(\kappa^2X\cdot Y-p)-\sqrt{(\kappa^2X\cdot Y-b)^2-(\kappa^2-1)(\kappa^2|Y|^2-p^2)}}{\kappa^2-1}.
  \end{equation}
We have similar existence results by replacing the assumptions (R1), (R2) by the following ones:
\begin{itemize}
 \item[(R3)] $\inf_{X\in\overline\Omega,Y\in\overline\Omega^*}X\cdot\frac{Y}{|Y|}\geq\frac{1}{\kappa}+\tau$, for some $0<\tau<1-\frac{1}{\kappa}$.
 \item[(R4)] Let $0<r_0<\frac{\kappa^2\tau^2}{4(\kappa-1)^2}\inf_{Y\in\Omega^*}|Y|$ and consider the cone in $\mathbb{R}^{n+1}$
  \[Q_{r_0}=\{tX\,:\,X\in\overline\Omega, 0<t<r_0\}.\]
 For each $\xi\in\mathbb{R}^n$ and for each $X\in Q_{r_0}$ we assume that $\overline\Omega^*\cap\{X+t\xi\,:\,t\geq0\}$ contains at most one point. That is, for each $X\in Q_{r_0}$ each ray emanating from $X$ intersects $\overline\Omega^*$ at most in one point. 
\end{itemize}

\end{remark}

\subsection{Near field refractor problem with parallel source}\label{ss55}

Recently, Guti\'errez and Tournier studied the parallel refractor problem \cite{GuT}, which can be described as follows. 
Suppose that a parallel light emits from $\Omega\subset\mathbb{R}^n\times\{0\}$ along $e_{n+1}=(0,\cdots,0,1)$ with positive intensity $f\in L^1(\Omega)$, $\Omega^*$ is a hypersurface in $\mathbb{R}^{n+1}$, which is referred to as the target domain. Suppose that $\Omega$ and $\Omega^*$ are surrounded by two homogeneous and isotropic media $I$ and $I\!I$, respectively. 
One seeks an optical surface $\mathcal{R}$ interface between media $I$ and $I\!I$, such that all rays refracted by $\mathcal{R}$ into medium $I\!I$ are received at the surface $\Omega^*$, and the prescribed radiation intensity received at each point $Y\in\Omega^*$ is $g(Y)$. 
Assume the energy conservation condition
  \begin{equation}\label{pr01}
   \int_\Omega f=\int_{\Omega^*}g.
  \end{equation}

Let $n_1, n_2$ be the indices of refraction of media $I$, $I\!I$, respectively, and $\kappa=n_1/n_2$. 
We assume that media $I\!I$ is denser than media $I$, that is, $\kappa<1$. The case when $\kappa>1$ can be treated in a similar way but the geometry of surface changes \cite{GH1,GuT}.
For simplicity, we assume that $\Omega^*\subset\{y_{n+1}=h\}$ for a constant $h>0$, and denote $Y=(y,h)$ for points on $\Omega^*$ and $X=(x,0)$ for points on $\Omega$. 

Consider the lower part of \emph{``inverse'' ellipsoid of revolution} with focus at $y\in\Omega^*$ and the axial direction $-e_{n+1}$. It has the uniform refracting property, namely all rays from the parallel source $\Omega$ along $e_{n+1}$ will be refracted to the focus points $y$. Explicitly, it is the graph of the function \cite{GuT}
  \begin{equation}
   \rho_{y,v}(x)=h-\frac{\kappa v}{1-\kappa^2}-\sqrt{\frac{v^2}{(1-\kappa^2)^2}-\frac{|x-y|^2}{1-\kappa^2}},
  \end{equation}
where $v$ is a constant satisfying $\frac{h(1-\kappa^2)}{\kappa}\leq v\leq \frac{h(1-\kappa^2)(1+\kappa)^2}{\kappa^3}$. The function $\rho_{y,v}$ is defined on the ball $B_{v/\sqrt{1-\kappa^2}}(y)$. 

If we regard $v=v(y)$ as a function on $\Omega^*$, we then have a family of ``inverse'' ellipsoids. 
Represent the refractor $\Gamma $ as graph $u|_\Omega$ for $u>0$, namely
  \[\Gamma_u=\{(x,u(x)) : x\in\Omega\}.\]
In an ideal system, at each point $(x,u(x))\in\Gamma_u$ on an \emph{admissible} refractor $\Gamma_u$ there exists a \emph{supporting ellipsoid}, i.e., for some $y\in\Omega^*$
  \begin{eqnarray*}
   u(x)\!\!&=&\!\!h-\frac{\kappa v(y)}{1-\kappa^2}-\sqrt{\frac{v(y)^2}{(1-\kappa^2)^2}-\frac{|x-y|^2}{1-\kappa^2}},\\
   u(x')\!\!&\leq&\!\!h-\frac{\kappa v(y)}{1-\kappa^2}-\sqrt{\frac{v(y)^2}{(1-\kappa^2)^2}-\frac{|x'-y|^2}{1-\kappa^2}}\quad\forall\, x'\in\Omega.
  \end{eqnarray*}

Similarly we can formulate this problem to a nonlinear optimisation problem \eqref{0001}--\eqref{prob}.
Set the functional
	\begin{equation}
		I(u,v) = \int_{\Omega}f(x)u + \int_{\Omega^*} g(y)\left(\frac{\kappa v(y)}{1-\kappa^2}+\sqrt{\frac{v(y)^2}{(1-\kappa^2)^2}-\frac{|T^{-1}y-y|^2}{1-\kappa^2}}\right),
	\end{equation}
and the constraint set 
	\[K=\{(u,v)\in C(\Omega)\times C(\Omega^*)\,:\,\phi(X,Y,u,v)\leq0\},\]
with the constraint function
	\begin{equation}
		\phi(X,Y,u,v)=u+\frac{\kappa v(y)}{1-\kappa^2}+\sqrt{\frac{v(y)^2}{(1-\kappa^2)^2}-\frac{|x-y|^2}{1-\kappa^2}}-h.
	\end{equation}
  
As in \cite{GuT} we need the following assumptions on the relative position of $\Omega$ and $\Omega^*$, which are due to the physical constraints for refraction:
 \begin{itemize}
 \item[(A)] There exists $0<\delta<1$ such that $\Omega\subset B_{\delta h}\sqrt{1-\kappa^2}/\kappa(y)$ for all $y\in\Omega^*$.
 \item[(B)] Set $M=h((1+\kappa)^3/\kappa^3-1)$. Assume that for all $x\in\Omega\times[-M,0]$ and for all $\gamma\in\mathbb{S}^n$, the ray $\{x+t\gamma : t>0\}$ intersects $\Omega^*$ in at most one point. 
\end{itemize}
The first condition is equivalent to the assumption that there exists $0<\beta<1$ such that $\langle-e_{n+1},\frac{X-Y}{|X-Y|}\rangle\geq\beta$ for all $Y\in\Omega^*$ and $x\in\Omega$ \cite{GuT}.

As in the previous example one can verify the hypotheses of Theorem \ref{t003}.
Therefore, by the proof of Theorem \ref{t003} and Remark \ref{rnl2} we have the existence result: 
\begin{corollary}
Assume that $f,g$ satisfy \eqref{pr01} and (A)--(B). Then for $x_0\in\Omega$ and $t\leq -\beta$ there exists a parallel refractor $u$ satisfying $u(x_0)=t$.
\end{corollary}
Note that the solutions need to be understood as generalised solutions as before. This existence result was previously obtained in \cite{GuT}, where they first consider the discrete case when the target is a set of points, then use an approximation to obtain the existence in the general case.

\begin{remark}
In addition to the examples arising in reflectors and refractors, there are many other nonlinear optimisation problems with potentials. For example, one can perturb the linear optimisation problem, such as the optimal transportation, to get a nonlinear one.
Moreover, similarly to \cite{MM} one can show that the objective functional of any solvable linear optimisation problem can be perturbed by a differentiable, convex or Lipschitz continuous nonlinear functional in such a way that (i) a solution of the original linear problem is a local or global solution of the perturbed nonlinear problem; (ii) each global solution of the perturbed nonlinear problem is also a solution of the linear problem. 
\end{remark}


\end{document}